\documentclass[11pt]{amsart}
\usepackage{amsmath}
\usepackage{amssymb}
\usepackage{tabularx}
\usepackage{enumerate}
\usepackage{graphicx}
\usepackage{texdraw}

\usepackage{mathrsfs}

\usepackage{amsfonts,amssymb,amsmath}
\usepackage{epsfig}

\makeatletter
\@addtoreset{equation}{section}

\makeatother
\newtheorem{thm}{Theorem}[section]
\newtheorem{lemma}[thm]{Lemma}

\newtheorem{remark}[thm]{Remark}
\newtheorem{defn}[thm]{Definition}
\newcommand{\R}{\mathbb{R}}

\begin{document}
\title[Convergence to the Grim Reaper]
      {Convergence to the Grim Reaper for a Curvature Flow with Unbounded Boundary Slopes$^*$}
\thanks{$^*$ This research was partly supported by NSFC (No.11671262, No.11871148) and Shanghai NSF in China (No. 17ZR1420900).}

\author[B. Lou, X. Wang, L. Yuan]{Bendong Lou$^{\dag}$,\ Xiaoliu Wang$^{\ddag}$ and Lixia Yuan$^{\dag,\S}$}
\thanks{$\dag$ Mathematics and Science College, Shanghai Normal University, Shanghai 200234, China.}
\thanks{$\ddag$ School of Mathematics, Southeast University, Nanjing 210018, China.}
\thanks{$^\S$ Corresponding author.}
\thanks{{\bf Emails:}  {\sf lou@shnu.edu.cn (B. Lou)}, {\sf xlwang@seu.edu.cn (X. Wang)}, {\sf yuanlixia@shnu.edu.cn (L. Yuan)}}
\date{}

 \subjclass[2010]{35K93, 53C44, 35C07}
 \keywords{Curvature flow, grim reaper, traveling wave, asymptotic behavior}

\maketitle

\begin{abstract}
We consider a curvature flow $V=H$ in the band domain $\Omega :=[-1,1]\times \R$, where, for a graphic curve $\Gamma_t$,
$V$ denotes its normal velocity and $H$ denotes its curvature. If $\Gamma_t$ contacts the two boundaries $\partial_\pm
\Omega$ of $\Omega$ with constant slopes, in 1993, Altschular and Wu \cite{AW1} proved that $\Gamma_t$ converges to
a {\it grim reaper} contacting $\partial_\pm \Omega$ with the same prescribed slopes. In this paper we consider the case where $\Gamma_t$
contacts $\partial_\pm \Omega$  with slopes equaling to $\pm 1$ times of its height. When the curve moves to infinity,
the global gradient estimate is impossible due to the unbounded boundary slopes. We
first consider a special symmetric curve and derive its uniform interior gradient estimates by using the zero number argument,
and then use these estimates to present uniform interior gradient estimates for general non-symmetric curves, which  lead
to the convergence of the curve in $C^{2,1}_{loc} ((-1,1)\times \R)$ topology to the {\it grim reaper} with span $(-1,1)$.
\end{abstract}


\section{Introduction}

Since the pioneering work of Gage and Hamilton \cite{GageH} in 1986, the curve shortening problem has been studied widely in different settings. An important case is to investigate the problem with some boundary conditions.  Those kind of boundary problems are not only interesting from geometric viewpoint, but also important in the study of phase transition models. Also, they can be used to describe the interface propagation arising in some reaction-diffusion equations, such as the Allen-Cahn equation. Usually, there will be a contact angle condition
at the intersection points between the evolving curves (interfaces) and the domain boundaries (see, for example, \cite{ORS, RSK}). In particular, we focus on the following curvature flow
\begin{equation}
  V=H \qquad \text{on} \qquad\Gamma_t\subset \Omega, \label{mcf}
\end{equation}
in the band domain $\Omega:=\{(x,y)|-1\leq x\leq 1, y\in\R\}$ in $\R^2$, where, $\Gamma_t$ is a family of embedded curves
in $\Omega$ which contact the boundaries $\partial_\pm \Omega := \{\pm1\}\times\R $ of $\Omega$ with prescribed angles, $V$ and $H$ denote the normal velocity and the curvature of $\Gamma_t$, respectively. Indeed, Huisken \cite{Hui} considered this curvature flow in 1998. He discovered a distance comparison principle to show that the flow with fixed endpoints (Dirichlet boundary condition) exists globally and converges to a straight line linking two endpoints, and the flow having orthonormal boundary contact (homogeneous Neumann boundary condition) also exists globally but converges to a constant eventually.

In case, for each $t\geq 0$, $\Gamma_t$ is the graph of a function $y=u(x,t)$, we have
$$
  V=\frac{u_t}{\sqrt{1+u^2_x}},\qquad H=\frac{u_{xx}}{\left(1+u^2_x\right)^{3/2}},
$$
and the problem can be expressed as
\begin{equation}\label{A}
\left\{
 \begin{array}{ll}
\displaystyle u_t=\frac{u_{xx}}{1+u^2_x}, & -1<x<1,\ t>0,\\
u_x(\pm 1,t)=g_\pm ,& t>0,
 \end{array}
 \right.
\end{equation}
where $g_-, g_+$ denote the boundary contact conditions and will be made clear later.
 In 1993, Altschuler and Wu \cite{AW1} studied the inhomogeneous Neumann
problem, that is, the case where $g_+, -g_-$ are positive constants. They
proved that any global solution converges to a {\it grim reaper} (which is also called a
traveling wave by some authors).  In 1994, Altschular and Wu  \cite{AW2} extended their result
in \cite{AW1} to two dimension. In 2012, Cai and Lou \cite{CaiLou2} considered \eqref{A}
with $g_\pm $ being (almost) periodic functions of $u$, and proved that any solution
converges to a (almost) periodic traveling wave. Recently, Yuan and Lou \cite{YuanLou} considered the case
where $g_\pm = g_\pm (u)$ are asymptotic periodic functions as $u\to \pm \infty$.
They constructed some entire solutions connecting two periodic traveling waves.

Other works related to the mean curvature flow \eqref{mcf}, as well as
its anisotropic analogues, in domains in $\R^2$ with boundaries include \cite{LMN, MNL} for problems in band domain with undulating boundaries;
\cite{CGK, GGH, GH, Koh} for self-similar solutions in sectors; \cite{ChenGuo, GMSW}
for problems on the half plane, etc.

\medskip

In all of the above mentioned works, the boundary slopes are bounded, no matter they are linear or nonlinear.
In this paper, we consider a different case:
\begin{equation}\label{B}
\left\{
 \begin{array}{ll}
\displaystyle u_t=\frac{u_{xx}}{1+u^2_x}, &   -1<x<1,\ t>0,\\
u_x(\pm1,t)=\pm u(\pm1,t), &  t>0,\\
u(x,0)=u_0(x), & -\leq x\leq 1,
 \end{array}
 \right.
\end{equation}
which has {\it unbounded boundary slopes} when $u$ goes to infinity. This kind of boundary conditions are of
interest not only in geometry and partial differential equations but also in some applied fields.
In 2012, Chou and Wang \cite{CW} considered the equation in \eqref{B} with Robin boundary conditions:
$$
u_x(\pm 1,t) = \alpha_\pm u(\pm 1,t) +\beta_\pm,\quad t>0.
$$
They divided the parameters $\alpha_\pm$ and $\beta_\pm$ into several cases, and
studied the asymptotic behavior in each cases. For the cases $\min u\to \infty$
or $\max u\to -\infty$ (this happens, for example, when $\alpha_- <0 <\alpha_+$), however,
the authors did not obtain the convergence of the solution and left it as an open problem.
In this paper, we will show the convergence to a {\it grim reaper} for unbounded solutions to \eqref{B}.
Our approach (with slight modification) and our results remain valid for the problems with
other unbounded boundary slopes such as the Robin ones, $u_x (\pm 1,t) = \pm u^2 (\pm 1, t)$,
$u_x(\pm 1,t) = \pm e^{u(\pm 1,t)}$, etc.

The global well-posedness of the problem \eqref{B} is studied in a standard way.
For any time-global solution moving upward to infinity, $u_x$ is unbounded due to the boundary conditions.
This will be the main difficulty in studying the asymptotic behavior for such solutions. It
is natural to consider the convergence in $L^\infty_{loc} ((-1,1))$ topology in such cases.
This, however, also needs some {\it uniform} (in $t$) interior gradient
estimates.  The well known results in this field as in \cite{ES} are not applied here since
they depend on the boundedness of $u$. Instead, we will use the so-called zero number argument
(i.e., zero number diminishing properties, cf. \cite{Ang, Lou2}) for one dimensional parabolic equations to derive the uniform
bounds for the gradient of the solution in any interior domain (see details in sections 4 and 5).
Furthermore, as can be expected, the profile of the solution might converge to a traveling wave with
infinite slopes near the boundaries, which should be the {\it grim reaper} with span in $(-1,1)$, that is,
\begin{equation}\label{0-gr}
\varphi_0 (x) + \frac{\pi}{2} t \quad \mbox{ with } \quad \varphi_0 (x) := -\frac{2}{\pi} \ln \left[\cos \Big( \frac{\pi x}{2} \Big)\right]\ \ (x\in (-1,1)).
\end{equation}
Actually we will show the following result.

\begin{thm}\label{thm:main}
Assume $u_0(x)\in C^1([-1,1])$ with
\begin{equation}
  u_0(x)\geq 1\ (x\in [-1,1]),\qquad u'_0(\pm1)=\pm u_0(\pm 1) .\label{1.4}
\end{equation}
Then the problem \eqref{B} has a time-global solution $u(x,t)$. It moves upward to
infinity and
\begin{equation}
  u(x,t+s) - u(0,s) \to \varphi_0(x) + \frac{\pi}{2} t, \qquad \mbox{as}\ \  s\rightarrow \infty,\label{5}
\end{equation}
in the topology of $C^{2,1}_{loc}\left( (-1,1)\times\R\right).$
\end{thm}

This paper is arranged as follows. In section 2, as preliminaries, we present traveling waves ({\it grim reaper}s)
contacting the boundaries of $\Omega$ with various different constant slopes. In section 3, we give some rough a
priori estimates and then show the time-global existence for the solution of \eqref{B}.
In section 4, we consider symmetric solutions of \eqref{B}. First we present
precise estimates for $u_x$ by using the zero-number argument, and then show its convergence to the {\it grim reaper}.
Finally,  in section 5 we consider general solutions of \eqref{B} which are not necessary symmetric.
By some further uniform interior estimates we show its convergence to the {\it grim reaper}.



\section{Traveling Waves}

First we give a definition.
\begin{defn}
A function $\underline{u}(x,t)$ satisfying
$$
\left\{
 \begin{array}{l}
\displaystyle \underline{u}_t\leq\frac{\underline{u}_{xx}}{1+\underline{u}^2_{x}},\qquad -1<x<1,\ t>0,\\
\underline{u}_x(1,t)\leq \underline{u}(1,t),\quad\underline{u}_x(-1,t)\geq -\underline{u}(-1,t),\qquad t>0,
 \end{array}
 \right.
$$
is called a lower solution of (\ref{B}). A function $\overline{u}(x,t)$ satisfying the reversed inequalities is called an upper solution
of \eqref{B}.
\end{defn}

As preliminaries, we study the following problem
\begin{equation}\label{C}
\left\{
 \begin{array}{ll}
\displaystyle u_t=\frac{u_{xx}}{1+u^2_x}, & -1<x<1,\quad t\in\R,\\
u_x(\pm1,t)=\pm h, & t\in \R.
 \end{array}
 \right.
\end{equation}
For each $h>0$, a traveling wave of (\ref{C}) (also called a translating solution in \cite{AW1})
is a special solution of the form
$$
u(x,t)=\varphi(x;h)+c(h)t.
$$
Substituting this formula into (\ref{C}) we easily obtain
\begin{equation}
\varphi(x;h) :=-\frac{1}{c(h)}\ln \left[\cos\left(c(h)x\right)\right].
\qquad c(h) :=\arctan h,
\end{equation}
$\varphi$ is called a {\it grim reaper} in \cite{AW1}.  Note that
\begin{equation}
\varphi(\pm1;h)= \frac{\ln(1+h^2)}{2\arctan h},\qquad \varphi_x(\pm1;h)=\pm h.
\end{equation}
Hence, for any $M\in\R$, $\varphi(x;h)+c(h)t+M$ is a lower solution of (\ref{B}) when
$$
h\leq\varphi(1;h)+c(h)t+M.
$$
It is an upper solution if the reversed inequality holds.

Besides the traveling waves of \eqref{C}, we have another {\it grim reaper} $\varphi_0(x)+\frac{\pi}{2}t$
with $\varphi_0$ defined by \eqref{0-gr}. Note that the definition domain of $\varphi_0(x)$ is $(-1,1)$,
that is, this {\it grim reaper} lies completely in $\Omega$. In what follows we will use the above
{\it grim reaper}s to give a priori estimates for the solution of (\ref{B}).



\section{Global Well-posedness of (\ref{B})}

Assume $u(x,t)$ is a classical solution of (\ref{B}) in the time-interval $[0,T]$ for some $T>0$.
We first give its $L^{\infty}$ estimate.
\begin{lemma}\label{bound-est}
 There exist $C_1,C_2>0$ with $C_2$ depending on $T$ such that
 \begin{equation}\label{3.1}
 c(1) t -C_1 \leq u(x,t)\leq C_2(T),\qquad x\in[-1,1],\ t\in[0,T].
 \end{equation}
\end{lemma}

\begin{proof}
Assume, for some $M_0>0$,
\begin{equation}\label{3.2}
1\leq u_0(x)\leq M_0,\qquad x\in[-1,1].
\end{equation}
Then $$
\underline{u}(x,t):=\varphi(x;1)+c(1)t+1-\varphi(1;1)
$$
is a lower solution of (\ref{B}) and so, by the comparison principle we have
\begin{equation}\label{3.3}
\underline{u}(x,t)=\varphi(x;1)+c(1)t+1-\varphi(1;1)\leq u(x,t),\qquad x\in[-1,1],\  t>0,
\end{equation}
which leads to the first inequality in (\ref{3.1}).

Next we consider the upper bound. Since
$$
h-\varphi(\pm1;h)=h-\frac{\ln(1+h^2)}{2\arctan h}\rightarrow \infty\qquad \text{as} \qquad h\rightarrow \infty,
$$
there exists $h=h_T$ large such that
\begin{equation}\label{3.4}
h_T>\varphi(\pm1;h_T)+\frac{\pi}{2}T+M_0.
\end{equation}
Set
$$
\overline{u}(x,t):=\varphi(x;h_T)+c(h_T)t+M_0, \qquad x\in[-1,1],\ t>0.
$$
We verify that $\overline{u}$ is an upper solution of (\ref{B}) in the time interval $[0,T]$.
In fact, $\overline{u}$ satisfies the equation in (\ref{B}). Moreover, for $t\in[0,T]$,
$$
\overline{u}(1,t) <  \varphi(1;h_T)+\frac{\pi}{2}t + M_0<h_T = \overline{u}_x(1,t),
$$
by (\ref{3.4}), and
$$
\overline{u}(x,0)=\varphi(x;h_T)+M_0\geq M_0\geq u_0(x), \qquad x\in[-1,1].
$$
Hence, $\overline{u}$ is an upper solution of (\ref{B}). By comparison principle we have
$$
u(x,t)\leq \overline{u}(x,t)\leq \varphi(1;h_T)+M_0+\frac{\pi}{2}T,
\qquad x\in[-1,1],\ t\in[0,T].
$$
This proves the second inequality of (\ref{3.1}).
\end{proof}

\bigskip
Next, we give the gradient estimate.
\begin{lemma}
Let $u(x,t)$ be a solution of \eqref{B} in the time interval $[0,T]$. Then there exist $C_3(T)$ such that
$$
\left|u_x(x,t)\right|\leq C_3(T), \qquad x\in[-1,1],\ t\in[0,T].$$
\end{lemma}

\begin{proof}
From the above lemma, we see that
$$
\left|u_x(\pm1,t)\right|=\left|u(\pm1,t)\right| \leq C_1 + C_2(T) , \qquad t\in[0,T].
$$
Using the maximum principle for $u_x$ we see that
$$
\left|u_x(x,t)\right|\leq C_3(T):=\max\{\|u'_0\|_C, C_1 + C_2(T)\}, \qquad x\in[-1,1],\ t\in[0,T].
$$
This proves the lemma.
\end{proof}

With the above a priori estimates in hand, by using the standard parabolic theory
we obtain the time-global existence of the classical solution $u(x,t)$.
Its uniqueness is proved in the standard way by using the maximum principle.




\section{Symmetric Solutions}

In this section we consider symmetric solutions. More precisely, we consider the case where $u_0(x)\in C^1([-1,1])$
satisfies the following conditions:
\begin{equation}\label{ini-1}
  u_0(x)=u_0(-x),\quad 1\leq u_0(x)\leq M_0, \quad 0< u'_0(x)<\varphi'_0 (x) \mbox{\ for } x\in (0,1),
\quad u'_0 (1)= u_0(1),
\end{equation}
for some $M_0>0$.  In this case we easily have
$$
u(x,t)= u(-x,t),\quad u_x(x,t)>0 \quad \mbox{ for } x\in (0,1],\ t>0.
$$
To study the convergence of $u$ (actually, the convergence of $u(x,t+s)-u(0,s)$ as $s\to \infty$),
we need further estimate for $u_x$. We will do this by the so-called zero number argument.

\subsection{Finer lower gradient estimate}
In this part, we show that \lq\lq the gradient of $u$ is not too small" for $x\in (0,1)$.

Fix an $h_0>0$. For any $h\geq M_0+h_0$ we see that
$$
\underline{u} (x,t;h) := \varphi(x;h_0) +c(h_0) t +h
$$
is a lower solution of (\ref{B}). Denote the union of the graphes of $\underline{u}(x,t;h)$ as
$$
\mathcal{D}(t) := \left\{(x,y)\mid y\geq \underline{u} (x,t; M_0+h_0),\ x\in[-1,1] \right\}.
$$
Then $\mathcal{D}(t)$ is the upper half of $\Omega$ with bottom $\left\{\left(x, \underline{u} (x,t; M_0 +h_0)\right) \mid x\in[-1,1]\right\}$,
which moves upward with speed $c(h_0)$.

We now construct another lower solution below the real solution $u$ such that it moves faster than $\mathcal{D}(t)$, and so
pushes $u$ entering the domain $\mathcal{D}(t)$ for large $t$. Then, by considering the derivatives of
$u(x,t)$ and $\underline{u} (x,t;h)$ at their contact points, we can obtain the desired gradient estimate.
To construct another lower solution below $u$, we notice by (\ref{3.1}) that $u(x,t)$ moves up to infinity.
Hence for any given $h^0>h_0$, there exists $t^0$ large such that
$$
u(x,t^0)>\varphi(x; h^0)+h^0\geq h^0,\qquad x\in[-1,1].
$$
Thus,
$$
\underline{u}^* (x,t):=\varphi(x;h^0)+h^0+c(h^0)t
$$
is also a lower solution to (\ref{B}). By the comparison principle we have
\begin{equation}\label{4.1}
\underline{u}^* (x,t)\leq u(x,t+t^0), \qquad x\in[-1,1],\ t\geq 0.
\end{equation}

Since $c(h^0)>c(h_0)$ we see that for all large $t$ (to say, $t\geq T^0$), $\underline{u}^* (x,t)$ rushes
into the domain $\mathcal{D}(t)$, so dose $u(x,t+t^0)$ due to (\ref{4.1}).
In other words, when $t\geq t^0+T^0$, the graph of $u(x,t)$ is immersed in the solid $\mathcal{D}(t)$.
Therefore, for any $t_1 \geq t^0+T^0$ and any $x_1\in (0,1)$, there exists a unique $h\geq M_0+h_0$
such that $u(x_1,t_1) = \underline{u} (x_1,t_1;h)$. We now show that
\begin{equation}\label{underlineu-x<u-x}
0< \underline{u}_x (x_1,t_1;h) \leq u_x(x_1, t_1).
\end{equation}
In this sense, we say that \lq\lq the gradient of $u$ is not too small" for $x\in (0,1)$.

We first consider the case where $u(x,t_1) \leq \underline{u} (x,t_1;h)$ or
$u(x,t_1) \geq \underline{u} (x,t_1;h)$ for $x\in (0,1)$. In such cases, it is clear that
$\underline{u}_x (x_1,t_1;h) =u_x(x_1, t_1)$.
Next, we consider the case
\begin{equation}\label{max>0>min}
\max_{x\in [0,1]} [u(x,t_1)-\underline{u} (x,t_1;h)] >0 > \min_{x\in [0,1]} [u(x,t_1)-\underline{u} (x,t_1;h)].
\end{equation}
By continuity, these inequalities remain valid for any $t$ near $t_1$. Hence such $t$ form an interval.
Denote the largest interval containing $t_1$ by $(s_1, s_2)$, then \eqref{max>0>min} holds for any $t\in (s_1,s_2)$
and $s_1 >0$ since $\underline{u} (x,0;h) \geq h >M_0 \geq u(x,0)$. At $s_1$ we have the following claim:

\medskip
\noindent
{\bf Claim 1}. $\underline{u}(x,s_1;h)\geq u(x,s_1)$ in $[0,1]$, and $\underline{u}(1,s_1;h)= u(1,s_1)$.

Totally, there are four cases between $\underline{u}(x,s_1;h)$ and $u(x,s_1)$:
\begin{itemize}
\item[(a)] $\underline{u}(x,s_1;h) > u(x,s_1) $ in $[0,1]$;
\item[(b)] $\underline{u}(x,s_1;h)\leq u(x,s_1)$ in $[0,1]$;
\item[(c)] \eqref{max>0>min} holds at $s_1$;
\item[(d)] $\underline{u}(x,s_1;h) \geq u(x,s_1)$ in $ [0,1]$,
{\it equality} holds at some points and {\it strict inequality} holds at other points.
\end{itemize}
We now show that the cases (a), (b) and (c) are impossible, and Claim 1 holds in case (d).
In fact, in case (a), by continuity we have $\underline{u}(x,t;h)> u(x,t)$ for $t$ satisfying $s_1 <t\ll s_1 +1$.
In case (b), by comparison we have $\underline{u}(x,t;h) < u(x,t)$ for $x\in [0,1]$
and $t>s_1$. In case (c), there is a small time interval containing $s_1$ as its interior point
such that \eqref{max>0>min} holds in this interval.  Thus, all of these three cases contradict the
definition of $(s_1, s_2)$.  The only left possible case is (d). If $\underline{u}(1,s_1;h) > u(1,s_1)$,
then using the strong comparison principle in $[0,1]\times [s_1, t)$
for $s_1 < t \ll s_1 +1$ we conclude that $\underline{u}(x,t;h) > u(x,t)$ in this domain.
This again contradicts the definition of $(s_1, s_2)$. Therefore, $\underline{u}(1,s_1;h) = u(1,s_1)$,
and Claim 1 is proved.

\medskip
\noindent
{\bf Claim 2}. For any $t\in (s_1, s_2)$, $\underline{u}(\cdot,t;h)$ contacts $u(\cdot,t)$ at exactly one point
$x= X(t,h)$ in $(0,1)$ with $\underline{u}_x (X, t;h) <  u_x (X, t)$.

Since, by Claim 1,
\begin{equation}\label{eta-slope0}
\underline{u}_x (1,s_1;h) =h_0 < \underline{u} (1,s_1;h) = u(1,s_1) = u_x (1,s_1),
\end{equation}
we have $\underline{u} (x,s_1;h) > u(x,s_1)$ in $(x_* ,1)$ for some $x_*$
near $1$. Using Claim 1 and the comparison principle in the domain $\{(x,t)\mid 0<x<x_*,\ t>s_1 \mbox{ with } t-s_1 \ll 1\}$
we conclude that $\underline{u} (x,t;h) > u(x,t)$ in this domain.
Therefore, Claim 2 holds for $s_1 < t\ll s_1 +1$, and the unique contact point between
$\underline{u} (\cdot,t;h)$ and $u(\cdot,t)$ is near $1$.
If we can show that $\underline{u} (1,t;h) < u(1,t)$ for all $t\in (s_1, s_2)$,
then combining with the Neumann condition at $x=0$:
$$
\underline{u}_x (0,t;h) = u_x( 0,t) =0,
$$
we can derive Claim 2 by the zero number argument (\cite{Ang}).
Assume by contradiction that for some $s_3\in (s_1, s_2)$, $\underline{u} (1,t;h) < u(1,t)$ for $t\in (s_1, s_3)$,
but $\underline{u} (1,s_3;h) = u(1,s_3)$.  Then Claim 2 holds in $(s_1, s_3)$.  Due to
$\underline{u} (1,s_3;h) = u(1,s_3)$ and the analogue of \eqref{eta-slope0} at $s_3$,
there exists $x^* \in (0, 1)$ such that $\underline{u} (x,s_3;h) > u(x,s_3)$ for $x\in [x^*, 1)$.
By continuity, there exists a small $\epsilon>0$ such that $\underline{u} (x^*, t;h) > u(x^*,t)$
for $t\in  (s_3 -\epsilon, s_3 +\epsilon)$.  Then, in the time period $(s_3 -\epsilon, s_3)$, the unique contact point
$X(t,h)$ between $\underline{u}(\cdot,t;h)$ and $u(\cdot,t)$  lies in $(x^*, 1)$ since
$\underline{u} (x^*,t;h) > u(x^*,t)$ and $\underline{u} (1,t;h) < u(1,t)$ for $t\in (s_3 -\epsilon, s_3)$.
Therefore, $\underline{u} (x,t;h) > u(x,t)$ for $(x,t)\in [0,x^*] \times (s_3-\epsilon, s_3)$.
Using the strong maximum principle we see that $\underline{u} (x,s_3;h) > u(x,s_3)$ for $x\in [0,x^*]$.
Consequently, $\underline{u} (x,s_3;h) > u(x,s_3)$ in $[0,1)$.
This, however, contradicts the definition of $(s_1, s_2)$ and our assumption $s_3 \in (s_1, s_2)$.
Therefore, Claim 2 holds for all $t\in (s_1, s_2)$.

Using Claim 2 in particular at $t_1\in (s_1, s_2)$, we have $\underline{u}_x (X(t_1,h),t_1;h) < u_x(X(t_1,h),t_1)$, where
$X(t_1, h)$ is the unique contact point between $\underline{u} (\cdot,t_1;h)$ and $u(\cdot,t_1)$,
which is nothing but $x_1$. This proves \eqref{underlineu-x<u-x}. Since $t_1 \geq t^0+T^0$ and $x_1\in (0,1)$
are arbitrarily given, we actually obtain the following lemma.
\begin{lemma}\label{lem:small-gradient}
For any $t\geq t^0+T^0$ and any $x\in (0,1)$, there holds
$$
u_x( - x, t) \leq \varphi'(-x;h_0)<0,\qquad 0< \varphi'(x;h_0) \leq u_x(x, t).
$$
\end{lemma}

The inequalities in this lemma in fact can be improved to strict ones by using the strong maximum principle
and the zero number argument, which is omitted here, since the present version is enough for our approach below.

\subsection{Finer upper gradient estimate}
We now show that \lq\lq the gradient of $u$ is not too large" for $x\in (0,1)$.

By our assumption $u'_0(x)<\varphi'_0(x)$ for $x\in (0,1)$, we see that $u_0(x) \leq \varphi_0(x) + u_0(0)$ in $(-1,1)$, and so
\begin{equation}\label{4.6}
u(x,t)<\varphi_0(x) + u_0(0)  +\frac{\pi}{2}t,\qquad x\in(-1,1),\ t>0
\end{equation}
by the comparison principle. On the other hand, for any $r < u_0(0)$,
$\varphi_0(x) + r - u_0(x)$ has exactly two non-degenerate zeros.
Since $\zeta(x,t):=\varphi_0(x)+\frac{\pi}{2}t + r -u(x,t)$ satisfies a linear
parabolic equation whose coefficients are bounded in any compact interval of $(-1,1)\times (0,\infty)$,
and since
$$
\zeta(-1+0,t) = \zeta(1-0, t) =+ \infty,\quad t>0,
$$
we can use the zero number argument to conclude that, for any $t>0$,  either
\begin{itemize}
\item[(1)] $\zeta (\cdot, t)$ has two non-degenerate zeros $\pm\rho(t)$ with $\rho(t)\in(0,1)$; or
\item[(2)] $\zeta  (\cdot, t)$ has a unique degenerate zero $0$; or
\item[(3)] $\zeta  (\cdot, t)$ is positive, and has no zeros in $(-1,1)$.
\end{itemize}

Note that, for any $t>0$,  the graph of $u(\cdot,t)$ is immersed in
$$
\mathcal{E}(t) := \left\{\left. \left( x, \varphi_0(x) + r +\frac{\pi}{2}t\right) \right| x\in (-1,1),\ r\leq u_0(0)\right\}.
$$
Hence, for any $t_2>0$ and any $x_2 \in (0,1)$, there exists a unique $r = R(t_2,x_2)<u_0(0)$ such that
$\zeta (\cdot, t_2)$ with $r=R(t_2,x_2)$ has exactly two zeros $x= \pm x_2$. Consequently,
\begin{equation}\label{not too large}
\varphi'_0(-x_2) < u_x(- x_2,t_2) < 0,\qquad 0< u_x(x_2,t_2) < \varphi'_0(x_2).
\end{equation}
In this sense, we say that \lq\lq the gradient of $u$ is not too large" for $x\in (0,1)$.

\subsection{Convergence of the solution}

Assume $u(x,t)$ is a symmetric solution starting from an initial data satisfying \eqref{ini-1}.
Let $\{t_n\}$ be a time sequence with $t_n\rightarrow \infty\ (n\rightarrow\infty)$. Set
$$
u_n(x,t):=u(x,t+t_n)-u(0,t_n),\qquad x\in[-1,1],\ -t_n<t<\infty.
$$
For any given small $\varepsilon>0$, by the results in the previous two subsections we have
\begin{equation}\label{5.1}
\varphi' (x;h_0)<u_{nx}(x,t)<\varphi'_0(x),\qquad x\in (0,1-\varepsilon],\ n\gg 1.
\end{equation}
$h_0>0$ can be as large as possible, provided $n$ is sufficiently large and $t$ is bounded. Combining with Lemma \ref{bound-est} we have
$$
\left\|u_n(x,t) \right\|_{C^{1,0}\left([\varepsilon-1,1-\varepsilon]\times [-T,T]\right)}\leq C_1(\varepsilon,T),
$$
for any $T>0$. By the $L^p$ estimates, Sobolev embedding theorem and the Schauder estimate we have,
for any $\alpha\in(0,1)$,
$$
\left\|u_n(x,t) \right\|_{C^{2+\alpha,1+\frac{\alpha}{2}}\left([\varepsilon-1,1-\varepsilon]\times [-T,T]\right)}\leq C_2(\varepsilon,T).
$$
Therefore, for any $\beta\in (0,\alpha)$, there exist  a subsequence $\{u_{n_i}\}$ of $\{u_n\}$ and a function
$\mathcal{U}_{T,\varepsilon}\in C^{2+\alpha,1+\frac{\alpha}{2}} $ $\left([\varepsilon-1,1-\varepsilon]\times [-T,T]\right)$
such that
$$
\left\|u_{n_i}-\mathcal{U}_{T,\varepsilon}\right\|_{C^{2+\beta,1+\frac{\beta}{2}}
\left([\varepsilon-1,1-\varepsilon]\times [-T,T]\right)}\rightarrow 0\quad  (i\rightarrow \infty).
$$
Using the Cantor's diagonal argument, we see that there exist a function $\mathcal{U}\in C^{2+\alpha,1+\frac{\alpha}{2}}_{loc}
\left((-1,1)\times \R\right)$ and a subsequence of $\{u_n\}$ (denoted it again by $\{u_{n_i}\}$) such that
 $$
u_{n_i}\rightarrow \mathcal{U}\ \ (i\to \infty), \quad\text{in}\quad C^{2+\beta,1+\frac{\beta}{2}}_{loc}\left((-1,1)\times \R\right) \quad \text{topology}.
 $$
Moreover, $\mathcal{U}(x,t)$ is an entire solution of the equation in (\ref{B}) with $\mathcal{U}(0,0)=0$ and, by (\ref{5.1}),
$$
\varphi' (x;h_0)\leq \mathcal{U}_x(x,t)\leq \varphi'_0(x),\qquad  x\in [0,1),\ t\in \R.
$$
Since $h_0>0$ can be as large as possible and since
$$
\varphi' (x;h_0)\rightarrow \varphi'_0(x) \mbox{ for } x\in (-1,1),\qquad\quad \text{as}\quad h_0\rightarrow \infty,
$$
we conclude that
$$
\mathcal{U}_x(x,t) =  \varphi'_0(x),\qquad x\in (-1,1),\ t\in \R,
$$
and so
$$
\mathcal{U}(x,t)= \varphi_0(x)+C(t),\qquad x\in (-1,1),\ t\in \R.
$$
Finally, since $\mathcal{U}$ is an entire solution of the equation in (\ref{B}), we see that
$C'(t)=\frac{\pi}{2} $ and thus
$$
\mathcal{U}(x,t)=\varphi_0(x)+\frac{\pi}{2}t,\qquad x\in(-1,1),\ t\in \R.
$$

The above result implies that $\{u_{n_i}\}$ converges to the special {\it grim reaper} $\varphi_0(x)+\frac{\pi}{2}t$.
Since this {\it grim reaper} is unique and the time sequence $\{t_n\}$ is arbitrarily given, we actually obtain the following result.

\begin{thm}\label{thm:converge-symmetric}
Assume $u(x,t)$ is the time-global solution of \eqref{B} with initial data $u_0(x)$ satisfying \eqref{ini-1}. Then,
for any $\alpha\in (0,1)$,
\begin{equation}\label{5.2}
u(x,t+s)-u(0,t)\rightarrow \varphi_0(x)+\frac{\pi}{2}t, \qquad \text{as} \ s\rightarrow \infty,
\end{equation}
in the $C^{2+\alpha,1+\frac{\alpha}{2}}_{loc}\left((-1,1)\times \R \right)$ topology.
\end{thm}

\section{General Solutions}
The conclusion in the previous section holds only for symmetric solutions.
We consider general solutions in this section, that is, we assume $u_0$
satisfies \eqref{1.4} in this section.

\subsection{Interior estimates}
We choose a smooth, even function $\psi (x)$ with
$$
0< \psi'(x)<\varphi'_0(x),\quad 0<\psi(x) <u_0(x),\quad \psi''(x)>0 \quad \mbox{for } x\in (0,1],\quad \psi'(1)=\psi(1).
$$
Then $\psi$ satisfies \eqref{ini-1} except for $\psi(x)\geq 1$. One example is $\psi(x)= \delta \Big[ \frac{\sqrt{2}}{2}
- \frac{4}{\pi +4} \cos \frac{\pi x}{4} \Big]$ for small $\delta>0$. Denote the solution of (\ref{B})
with $u(x,0)=\psi(x)$ by $u(x,t;\psi)$. Then it is symmetric, $u_{xx}(x,t;\psi) >0$ due to $u_t(x,t;\psi)>0$,
and it satisfies all the conclusions in the previous section. Furthermore, it moves upward monotonically to infinity,
so we have
$$
u_0(x) < u(x,T;\psi),\qquad x\in[-1,1],
$$
for some positive $T$. Thus, by the comparison principle we have
\begin{equation}\label{4.9}
u(x,t;\psi) < u(x,t; u_0)< u(x,t+T;\psi) ,\qquad x\in [-1,1],\ t>0.
\end{equation}
This formula gives the $L^\infty$ estimate for $u(x,t; u_0)$.

In what follows, we want to present a uniform interior gradient estimate.
First we prove a lemma.

\begin{lemma}\label{lem4.3}
For any small $\varepsilon\in(0, \frac12)$ and any $t>0$, there hold
\begin{equation}\label{4.10}
\min_{1-2\varepsilon\leq x\leq 1-\varepsilon}\left|u_x(x,t;u_0)\right|<M_1 :=\varepsilon^{-1} \Big[\varphi_0(1-\varepsilon)+\frac{\pi}{2}T \Big],\qquad
\min_{\varepsilon-1 \leq x\leq 2\varepsilon-1}\left|u_x(x,t;u_0)\right|<M_1.
\end{equation}
\end{lemma}

\begin{proof}
We only prove the first inequality since the second one is proved similarly.
Assume by contradiction that, for some $t=t_0>0$,
$$
\left|u_x(x,t_0;u_0)\right| \geq M_1, \qquad x\in[1-2\varepsilon, 1-\varepsilon].
$$
Integrating it over $[1-2\varepsilon, 1-\varepsilon]$ we obtain
\begin{equation}\label{4.11}
u\left(1-\varepsilon, t_0; u_0\right)-u\left(1-2\varepsilon, t_0; u_0\right)
 \geq \varphi_0\left(1-\varepsilon\right)+\frac{\pi}{2}T.
\end{equation}

On the other hand, by (\ref{4.9}) we have
\begin{equation}\label{4.12}
 \begin{array}{lll}
 u \left(1-\varepsilon, t_0; u_0\right)-u\left(1-2\varepsilon, t_0; u_0\right)
 & < & u\left(1-\varepsilon, t_0+T; \psi\right)-u\left(1-2\varepsilon, t_0; \psi\right) \\
 & \leq & u\left(1-\varepsilon, t_0+T; \psi\right)-u\left(0, t_0; \psi\right)
 \end{array}
\end{equation}
since $u(\cdot, t; \psi)$ is convex and symmetric.
By \eqref{not too large}, $u(x,t_0;\psi)<\varphi_0(x)+u(0,t_0;\psi)$, and so by comparison we have
$$
u(x,t_0 +T;\psi)<\varphi_0(x)+\frac{\pi}{2}T+u(0,t_0;\psi).
$$
In particular, at $x=1-\varepsilon$ we have
$$
u\left(1-\varepsilon, t_0+T; \psi\right)-u\left(0, t_0; \psi\right)< \varphi_0(1-\varepsilon)+\frac{\pi}{2}T.
$$
This contradicts (\ref{4.11}) and (\ref{4.12}). This proves the lemma.
\end{proof}

Using this lemma we can prove the following interior gradient estimates.

\begin{lemma}\label{lem:interior-g-est}
For any small $\varepsilon>0$, there exists $T_\varepsilon >0$ such that
\begin{equation}\label{interior-g-est}
 |u_x (x,t; u_0)| \leq M_2, \quad   -1+2\varepsilon <x <1-2\varepsilon,\ t>T_\varepsilon,
 \end{equation}
where $M_2 := \max\{M_1, \|u_x(\cdot, T_\varepsilon; u_0)\|_{L^\infty}\}$ and $M_1$ is that in \eqref{4.10}.
\end{lemma}

\begin{proof}
Since $u(x,t;u_0) \to \infty$ as $t\to \infty$, there exists $ T' >0$ large such that $u(\pm 1, t; u_0)
> M_1$ for all $t>T' $. Denote $\zeta(x,t):= u_x(x,t; u_0)-M_1$, then $\zeta$ solves
$$
\zeta_t = \frac{\zeta_{xx}}{1+u^2_x} - \frac{2 u_x }{(1+u^2_x)} \zeta^2_x,\quad -1<x<1,\ t>0,
$$
and $\zeta (1,t)>0>\zeta(-1,t)$ for $t>T'$. Using the zero number properties (cf. \cite{Ang})
we conclude that, for some $T_\varepsilon  >T' $, the function $\zeta(\cdot, t)$ has only non-degenerate zeros for
$t\geq T_\varepsilon $. Denote the largest zero of $\zeta(\cdot,t)$ in $(-1,1)$ by $ \rho_+(t)$. Due to
the non-degeneracy of $\rho_+(t)$ we see that $x= \rho_+(t)$ is a continuous curve. Moreover,
\eqref{4.10} in the previous lemma indicates that $\rho_+(t) > 1- 2\varepsilon$.
In a similar way we can find another continuous curve $x=\rho_-(t)$ for $t>T_\varepsilon $ ($T_\varepsilon $ can be chosen larger if necessary),
with $\rho_-(t) \in (-1, -1 + 2\varepsilon)$, such that $u_x(\rho_-(t),t) = - M_1$ for $t >T_\varepsilon$.
Then, using the maximum principle for $u_x$ in the domain $D(T_\varepsilon ) := \{(x,t) \mid \rho_-(t)<x<\rho_+(t),\ t>T_\varepsilon \}$
we conclude that $|u_x (x,t;u_0)|\leq M_2$ in $D(T_\varepsilon)$. The estimate \eqref{interior-g-est} then follows
from the fact $\rho_-(t)<-1+2\varepsilon < 1-2\varepsilon < \rho_+(t)$ for $t>T_\varepsilon$.
\end{proof}

\subsection{Convergence of general solutions}

Let $\{t_n\}$ be any time sequence with $t_n\to \infty$, we consider the solution sequence $\{ u(x,t+t_n; u_0)-u(0,t_n;\psi)\}$.

For any given small $\varepsilon>0$ and any $\tau >0$, let $T_\varepsilon$ be as that in Lemma \ref{lem:interior-g-est},
then \eqref{4.9} and \eqref{interior-g-est} imply that, for all large $n$, $u(x,t+t_n; u_0)-u(0,t_n;\psi)$ is bounded
in $C^{1,0}([2\varepsilon-1, 1-2\varepsilon]\times [-\tau, \tau])$ norm, and the bounds are independent of $n$.
Using the standard parabolic theory we can even have the $C^{2+\alpha,1+\frac{\alpha}{2}}
([2\varepsilon-1, 1-2\varepsilon] \times [-\tau, \tau])$ (for any $\alpha\in (0,1)$) bounds for the solution sequence, also uniform in $n$, and
so we can find a convergent subsequence.
Taking $\varepsilon \to 0$ and $\tau \to \infty$, and using the Cantor's diagonal argument we conclude that,
there is a subsequence of $\{t_n\}$, denoted is again by $\{t_n\}$, such that
\begin{equation}\label{u-n-to-W}
u(x,t+t_n; u_0)-u(0,t_n;\psi)\rightarrow \mathcal{W}(x,t) \qquad \text{in}\quad
C^{2+\alpha,1+\frac{\alpha}{2}}_{loc}\left((-1,1)\times \R \right)\quad \text{topology},
\end{equation}
for some entire solution $\mathcal{W}$ of the equation in (\ref{B}).
On the other hand, as a consequence of Theorem \ref{thm:converge-symmetric} we have, as $n\to \infty$,
$$
u(x,t+t_n;\psi)-u(0, t_n;\psi)\rightarrow \varphi_0(x)+\frac{\pi}{2}t, \qquad u(x,t+T+t_n;\psi)
-u(0,t_n;\psi)\rightarrow \varphi_0(x)+\frac{\pi}{2}(t+T).
$$
Hence, we conclude from (\ref{4.9}) that
\begin{equation}\label{5.6}
\varphi_0(x)+\frac{\pi}{2}t\leq\mathcal{W}(x,t)\leq \varphi_0(x)+\frac{\pi}{2}t+\frac{\pi}{2}T,\qquad x\in(-1,1),\ t\in \R.
\end{equation}

Denote
$$
\theta(x,t):= \arctan u_x(x,t),\quad x\in [-1,1],\ t>0.
$$
Then $\theta$ satisfies
$$
\left\{
 \begin{array}{ll}
 \theta_t = \cos^2 \theta \cdot \theta_{xx}, & -1<x<1,\ t>0,\\
 \theta(\pm 1,t)=  \pm \arctan u(\pm 1,t), & t>0.
 \end{array}
 \right.
$$
The global existence of $u$ implies that $\theta$ exists for all $t>0$. For any $T>0$, by Lemma \ref{bound-est} we have
\begin{equation}\label{bound-theta}
|\theta(x,t)|\leq \arctan [C(T)],\quad x\in [-1,1],\ t\in [0,T],
\end{equation}
and $\theta(\pm 1,t) = \pm \arctan u(\pm 1,t) \to \pm \frac{\pi}{2}$ as $t\to \infty$.
Moreover, by \eqref{u-n-to-W} we have
$$
\theta(x,t_n +t) \to \Theta(x,t) := \arctan [\mathcal{W}_x (x,t)] \mbox{ in } C^{2+\alpha, 1+\alpha/2}_{loc} ((-1,1)\times \R)
\mbox{ topology}.
$$

1. We claim that $\Theta(x,t)$ is a stationary solution of $\theta_t = \cos^2 \theta \cdot \theta_{xx}$. Without loss of generality
we only need to prove that
$\Theta(x,0)$ is a stationary one. Denote $a:= \Theta_x (0,0),\ b:= \Theta(0,0)$ and $v(x):= ax+b$,
then $v(x)$ is a stationary solution of $\theta_t = \cos^2 \theta \cdot \theta_{xx}$, and the function
$\xi (x,t) := \theta (x,t) -v(x)$ satisfies one of the following boundary conditions:
$$
\begin{array}{ll}
 \mbox{in case }  v(-1) >-\frac{\pi}{2},\ v(1) \geq \frac{\pi}{2}: & \xi(-1,t) < 0,\ \xi (1,t)<0 \mbox{ for all large } t;\\
  \mbox{in case }  v(-1) >-\frac{\pi}{2},\ v(1) < \frac{\pi}{2}:  & \xi(-1,t) < 0,\ \xi(1,t)>0 \mbox{ for all large } t;\\
  \mbox{in case }   v(-1) \leq -\frac{\pi}{2},\ v(1) \geq \frac{\pi}{2}: & \xi(-1,t) > 0,\ \xi(1,t)<0 \mbox{ for all large } t;\\
  \mbox{in case }   v(-1) \leq -\frac{\pi}{2},\ v(1) < \frac{\pi}{2}: & \xi(-1,t) > 0,\ \xi(1,t)>0 \mbox{ for all large } t.
\end{array}
$$
Moreover, $\xi$ satisfies the linear equation $\xi_t = \cos^2 \theta\cdot \xi_{xx}$.
Hence the zero number argument is applied to conclude that $\xi(x,t)$ has no degenerate zeros for all large $t$.
Then using a similar argument as in the proof of \cite[Lemma 2.6]{DM} we conclude that the $\omega$-limit $\Theta(x,t)-v(x)$
of $\theta (x,t)-v(x)$ either satisfies (1) $\Theta(x,t)\equiv v(x)$, or satisfies (2) $\Theta(x,t)\not \equiv v(x)$ and
$\Theta(x,t)-v(x) $ has no degenerate zeros for each $t\in \R$. Case (2), however,
contradicts the definition of $v(x)$. Therefore, only case (1) holds. This prove our claim.

2. Next we show that $\Theta (x,t) \equiv \frac{\pi}{2}x$. In the previous step, we have shown that
$\Theta(x,t)\equiv v(x)\equiv ax +b$ for some  $a,b\in \R$.
If $v(\pm 1)= \pm \frac{\pi}{2}$, then the conclusion is proved.
On the contrary, we assume, without loss of generality that, for some $x_0\in (0,1)$ and some small $\delta >0$, there holds
$\Theta(x,t)\equiv v(x) <\frac{\pi}{2}x- \pi \delta$ in $x\in [x_0,1)$.
For any given small $\epsilon>0$, since
$|\theta(x,t_n) - \Theta(x,0)|\leq \frac12 \pi \delta$ in $[-1+\epsilon, 1-\epsilon]$
when $n$ is sufficiently large, we have
$$
 \arctan u_x (x,t_n) = \theta (x,t_n) \leq \Theta(x,0) + \frac12 \pi \delta = v(x) + \frac12 \pi \delta
 < \frac{\pi}{2} (x - \delta),\quad x\in [-1+\epsilon, 1-\epsilon].
$$
Integrating the above inequalities over $[x_0,1-\epsilon]$ we have
$$
u(1-\epsilon,t_n) \leq u(x_0,t_n) + \int_{x_0}^{1-\epsilon} \tan \Big[\frac{\pi}{2} (x -\delta) \Big] dx= u(x_0, t_n) +
\varphi_0 (1-\epsilon-\delta)-\varphi_0(x_0-\delta).
$$
Taking limit as $n\to \infty$ we have
$$
\mathcal{W}(1-\epsilon,0)\leq \mathcal{W}(x_0,0) + \varphi_0 (1-\epsilon-\delta)-\varphi_0(x_0-\delta).
$$
Combining with \eqref{5.6} we have
$$
\varphi_0(1-\epsilon) \leq \mathcal{W}(x_0,0) + \varphi_0 (1-\epsilon-\delta)-\varphi_0(x_0-\delta).
$$
This is a contradiction as $\epsilon\to 0$. This proves our claim, and so $\Theta(x,t)\equiv \frac{\pi x}{2}$.

3. As an $\omega$-limit of $\theta$, the function $\Theta(x,t)\equiv \frac{\pi}{2}x$ is unique. Hence, $\theta(x,t)\to \frac{\pi}{2}x$ as
$t\to \infty$. Equivalently, $\mathcal{W}_x (x,t)\equiv \tan [\frac{\pi}{2}x]$, and so
$$
\mathcal{W}(x,t) = \varphi_0(x) + K(t).
$$
Since $\mathcal{W}(x,t)$ is an entire solution, by its equation we have
$\mathcal{W}_t(x,t)\equiv \frac{\pi}{2}$, and so $K(t) = \frac{\pi}{2}t +K_0$
for some $K_0 \in \R$. By \eqref{5.6} we actually have $ K_0 \in [0, \frac{\pi}{2}T]$.
In summary,
$$
\mathcal{W}(x,t) = \varphi_0(x)+\frac{\pi}{2}t+ K_0, \qquad x\in(-1,1),\ t\in \R.
$$

\medskip
\noindent
{\it Proof of Theorem} \ref{thm:main}. From above discussion we see that, as $n\to \infty$,
$$
u(x,t+t_n; u_0) - u(0,t_n;\psi) \to \varphi_0(x) +\frac{\pi}{2} t + K_0,
$$
and so
\begin{eqnarray*}
u(x,t+t_n;u_0) - u(0,t_n;u_0) & = & [u(x,t+t_n;u_0) - u(0,t_n;\psi)] - [u(0,t_n;u_0) - u(0,t_n;\psi)]\\
 & \to & \varphi_0(x) + \frac{\pi}{2} t .
\end{eqnarray*}
Since $\{t_n\}$ is an arbitrarily chosen time sequence tending to infinity, we conclude that
$$
u(x,t+s;u_0) - u(0,s;u_0) \to \varphi_0 (x) +\frac{\pi}{2} t,\quad \mbox{ as } s\to \infty.
$$
This proves Theorem \ref{thm:main}. \hfill $\Box$

\begin{remark}\rm
We see from Sections 4 and 5 that the zero number argument plays a key role in the {\it uniform}
interior gradient estimates, which does not rely on the boundedness of $u$ and can be
extended to other problems with unbounded solutions (but only for one dimensional problems).
\end{remark}


\end{document}